\documentclass{amsart}
\usepackage{amssymb,amsmath,amsthm}
\title{A topos for a nonstandard functional interpretation}
\author{Benno van den Berg}
\date{January 16, 2013}
\usepackage{ast2}
\usepackage[margin=1.4in]{geometry}

\usepackage{ast2}

\begin{document}

\maketitle

\begin{abstract}
We introduce a new topos in order to give a semantic account of the nonstandard functional interpretation introduced by Eyvind Briseid, Pavol Safarik and the author.
\end{abstract}

\section{Introduction}

The aim of this short note is to give a semantic, topos-theoretic account of the nonstandard functional interpretation which the author, together with Eyvind Briseid and Pavol Safarik, introduced in \cite{bergetal12}, thus answering a question the author left open in \cite{berg12}. In this way this note is similar to the author's paper on the Herbrand topos \cite{berg12}, which did the same for Herbrand realizability, a realizability interpretation we also introduced in \cite{bergetal12}. Indeed, a good way to think about the topos to be defined here is as a Herbrandized version of the modified Diller-Nahm topos (for which see \cite{streicher06, biering08}).

\section{Notation}

Let us first establish some notation. We assume that we have fixed some pairing function, coding pairs of natural numbers as natural numbers. We will not distinguish notationally between pairs and codes of pairs and write $(n, m)$ for both the pair consisting of $n$ and $m$ and its code. Also, Kleene application will be written as ordinary application, so the result of applying the $n$th recursive function to the argument $m$ is written as $n(m)$, whenever it is defined.
 
For $X, Y \in {\rm Pow}(\NN)$, we will write
\begin{eqnarray*}
X \times Y & = & \{ (x, y) \in \NN \, : \, x \in X, y \in Y \}, \\
X + Y & = & \{ (0, x) \, : x \in X \} \cup \{ (1, y) \, : \, y \in Y \}, \\
X \to Y & = & \{ a \in \NN \, : \, (\forall x \in X) \, a(x) \mbox{ is defined and } a(x) \in Y \},
\end{eqnarray*}
as usual. In addition, we will write
\[ X^* = \{ a \in \NN \, : \, a \mbox{ codes a finite set all whose elements belong to } X \}. \]
Note that the empty set always belongs to $X^*$. We will use common set-theoretic notation when manipulating elements of $X^*$. 

We will always regard $X^*$ as a (pre)order, ordered by inclusion. Also note that we have an ``exponential isomorphism'' $(X+Y)^* \cong X^* \times Y^*$, which is not just a bijection, but also an order-isomorphism (if we order $X^* \times Y^*$ in the standard way). In what follows, we will often implicitly use this isomorphism and regard elements of $(X + Y)^*$ as pairs $(a, b)$ with $a \in X^*$ and $b \in Y^*$.

It will also be convenient to introduce the following piece of notation: if $x \in (S \to T^*)^*$ and $y \in S$, then we will write
\[ x[y] := \bigcup_{z \in x} z(y) \in T^*. \]
Another thing which we often implicitly use is that $x \subseteq x'$ implies $x[y] \subseteq x'[y]$ for all $y$.

\section{Definition of the tripos}

We define an preorder indexed over the category of sets and then show it is a tripos. First of all, we put
\begin{eqnarray*}
\Sigma_{st} & = & \{ (X, Y, R) \in {\rm Pow}(\NN)^2 \times {\rm Pow}(\NN \times \NN) \, : \, R \subseteq X^* \times Y \\
& & \mbox{ and } (\forall x, x' \in X^*, y \in Y) \, (x, y) \in R, x \subseteq x' \to (x', y) \in R \}.
\end{eqnarray*}
For $p = (X, Y, R)\in \Sigma_{st}$ we will  write 
\begin{eqnarray*}
p^+ & = & X, \\
p^{++} & = & X^*, \\
p^- & = & Y, \\
p(x, y) & = & R(x, y),
\end{eqnarray*}
respectively.

\begin{defi}{Dstttripos}
For any set $I$ the preorder above $I$ consists of functions $I \to \Sigma_{st}$. We write $\vdash_I$ for its preorder structure and we will have $\varphi \vdash_I \psi$ iff there exist
\begin{eqnarray*}
e^+ & \in & \bigcap_{i \in I} \varphi_i^{++} \to \psi_i^{++} \\
e^- & \in & \bigcap_{i \in I}  \varphi_i^{++} \times \psi_i^- \to (\varphi_i^-)^*
\end{eqnarray*}
such that
\[ \forall i \in I, a \in \varphi_i^{++}, b \in \psi_i^- \, [ \, \forall c \in e^-(a, b) \, \varphi_i(a, c) \to \psi_i(e^+(a), b)]. \] 
Reindexing is simply given by precomposition.
\end{defi}

\begin{lemm}{indexedpreorder} This defines an indexed preorder.
\end{lemm}
\begin{proof}
$p \vdash p$ is realized by $e^+(x) = x, e^-(x, y) = \{ y \}$. In addition, if $(e^+, e^-)$ realizes $p \vdash q$ and $(f^+, f^-)$ realizes $q \vdash r$, then $p \vdash r$ is realized by $(g^+, g^-)$ with $g^+(x) = f^+(e^+(x)), g^-(x, z) = \bigcup_{y \in f^-(e^+(x), z)} e^-(x, y)$. The preorder structure is obviously stable along reindexing.
\end{proof}

\begin{theo}{istripos}
The indexed preorder defined above is a tripos.
\end{theo}

We will call the associated topos the \emph{$D_{st}$-topos} and denote it by ${\bf Dst}$. The following sequence of lemmas will prove \reftheo{istripos}.

\begin{lemm}{truthandfalsity}
Truth is given by $(\emptyset, \emptyset, \emptyset)$ and falsity by $(\emptyset, \{ 0 \}, \emptyset)$.
\end{lemm}

\begin{lemm}{products}
The conjunction $p \land q$ is given by
\begin{eqnarray*}
(p \land q)^+ & = & p^+ + q^+, \\
(p \land q)^- & = & p^- + q^-, \\
(p \land q)((n, m), (i, k)) & \Leftrightarrow & \big( i = 0 \land p(n, k) \big) \mbox { or } \big( i = 1 \land q(m, k) \big).
\end{eqnarray*}
\end{lemm}
\begin{proof}
Note that we have used the exponential isomorphism $(X + Y)^* \cong X^* \times Y^*$ in order to identify $(p \land q)^{++}$ with $p^{++} \times q^{++}$. We will keep on making this identification.

The projection $p \land q \vdash p$ is realized by $e^+(a, b) = a$ and $e^-((a,b), c) = (\{ c \}, \emptyset)$, while $p \land q \vdash q$ is realized by $e^+(a, b) = b$ and $e^-((a,b),c) = (\emptyset, \{ c \})$.

Now suppose $r \vdash p$ is realized by $(e^+, e^-)$, while $r \vdash q$ is realized by $(f^+, f^-)$. Then $r \vdash p \land q$ is realized by $g^+(x) = (e^+(x), f^+(x))$ and $g^-(x, (0, y)) = e^-(x,y)$ and $g^-(x,(1,y)) = f^-(x, y)$.
\end{proof}

\begin{lemm}{sums}
The disjunction $p \lor q$ is given by
\begin{eqnarray*}
(p \lor q)^+ & = & p^+ + q^+, \\
(p \lor q)^- & = & p^- \times q^-, \\
(p \lor q)((n, m), (k, l)) & \Leftrightarrow & p(n, k) \mbox { or } q(m, l).
\end{eqnarray*}
\end{lemm}
\begin{proof}
Again, we identify $(p \lor q)^{++}$ with $p^{++} \times q^{++}$.

First, the inclusions. $p \vdash p \lor q$ is realized by $e^+(x) = (x, \emptyset)$ and $e^-(x, (y, z)) = \{ y \}$, while $q \vdash p \lor q$ is realized by $e^+(x) = (\emptyset,x)$ and $e^-(x, (y, z)) = \{ z \}$.

Now suppose $p \vdash r$ is realized by $(e^+, e^-)$, i.e.,
\[  \forall a \in p^{++}, b \in r^- \, [ \, \forall c \in e^-(a, b) \, p(a, c) \to r(e^+(a), b)], \]
while $q \vdash r$ is realized by $(f^+, f^-)$, i.e.,
\[  \forall a \in q^{++}, b \in r^- \, [ \, \forall c \in f^-(a, b) \, q(a, c) \to r(f^+(a), b)]. \]
Then, we claim, $p \lor q \vdash r$ is realized by $g^+(x, y) = e^+(x) \cup f^+(x)$ and $g^-((x, y), z) = \{ (s,t) \, : \, s \in e^-(x,z), t \in f^-(y,z) \}$. Because we have for all $x \in p^{++}, y \in q^{++}, z \in r^-$ that:
\begin{eqnarray*}
\forall (s,t) \in g^-((x, y), z) \, \big( \, p(x, s) \lor q(y, t) \, \big) & \to  \\
\forall s \in e^-(x,z), t \in f^-(y,z) \, \big( \, p(x, s) \lor q(y, t) \, \big) & \to & \mbox{(intuitionistic logic)} \\
\forall s \in e^-(x,z) \, p(x, s) \lor \forall t \in f^-(y,z) \, q(y, t)  & \to  \\
r(e^+(x), z) \lor r(f^+(y), z) & \to & \mbox{(upwards closure in first component)} \\
r(g^+(x, y), z).
\end{eqnarray*}
\end{proof}

\begin{lemm}{exponentials}The implication $p \to q$ is given by
\begin{eqnarray*}
(p \to q)^+ & = & (p^{++} \to q^{++}) + (p^{++} \times q^- \to (p^-)^* ) \\
(p \to q)^- & = & p^{++} \times q^-, \\
(p \to q)((e^+, e^-), (a, b)) & \Leftrightarrow & \big( \, \forall c \in e^-[(a, b)]p(a, c) \, \big)\to q(e^+[a], b).
\end{eqnarray*}
\end{lemm}
\begin{proof}
Suppose $(e^+, e^-)$ realizes $r \land p \vdash q$. Then $r \vdash (p \to q)$ is realized by 
\begin{eqnarray*}
f^+(x) & = & (\{\lambda y. e^+(x, y)\}, \{\lambda y,z. \pi_2 e^-((x, y), z) \}), \\
f^-(x, (y, z)) & = & \pi_1 e^-((x, y), z).
\end{eqnarray*}
Conversely, if $(e^+, e^-)$ realizes $r \vdash (p \to q)$, then $r \land p \vdash q$ is realized by:
\begin{eqnarray*}
f^+(x, y) & = & (\pi_1e^+(x))[y], \\
f^-((x, y), z) & = & (e^-(x, (y, z)), (\pi_2e^+(x))[(y, z)]).
\end{eqnarray*}
\end{proof}

\begin{lemm}{universalquantification} For $u: I \to J$ and $\varphi: I \to \Sigma_{st}$ universal quantification is given by:
\begin{eqnarray*}
\forall_u(\varphi)_j^+ & = & \bigcap_{i \in I} \, [u(i) = j] \to \varphi_i^{++} \\
\forall_u(\varphi)_j^- & = & \bigcup_{i \in u^{-1}(j)} \varphi^-_i \\
\forall_u(\varphi)_j(a, b) & \Leftrightarrow & (\forall i \in u^{-1}(j)) \, \big( \, b \in \varphi_i^{-} \to \varphi_i(a[0], b)\, \big).
\end{eqnarray*}
Here $[i = j] = \{ 0 \, : \, i = j \}$. Also the Beck-Chevalley condition holds.
\end{lemm}
\begin{proof}
Suppose $\varphi: I \to \Sigma_{st}$ and $\psi: J \to \Sigma_{st}$. We have to show the equivalence of the following two statements:
\begin{itemize}
\item[(a)] $\psi \vdash_J \forall_u(\varphi)$, i.e., there exist
\[ e^+ \in \bigcap_{j \in J} \psi_j^{++} \to \forall_u(\varphi)_j^{++}  \qquad \mbox{ and } \qquad e^- \in \bigcap_{j \in J} \psi_j^{++} \times \forall_u(\varphi)_j^-  \to (\psi_j^-)^* \]
such that
\[ \forall j \in J, a \in \psi_j^{++}, b \in \forall_u(\varphi)^-_j \big( \, \forall c \in e^-(a, b) \, \psi_j(a, c) \big) \to \forall_u(\varphi)_j (e^+(a), b). \]
\item[(b)] $u^* \psi \vdash_I \varphi$, i.e., there exist 
\[ f^+ \in \bigcap_{i \in I} \psi_{u(i)}^{++} \to \varphi_i^{++} \qquad \mbox{ and } \qquad f^- \in \bigcap_{i \in I} \, \psi_{u(i)}^{++} \times \varphi_i^- \to (\psi_{u(i)}^-)^* \]
such that
\[ \forall i \in I, a \in \psi_{u(i)}^{++}, b \in \varphi_i^- \big( \, \forall c \in f^-(a, b) \, \psi_{u(i)}(a, c) \, \big)  \to \varphi_i(f^+(a), b). \]
\end{itemize}
(a) $\Rightarrow$ (b): Take $f^+(x) = e^+(x)[0]$ and $f^-(x, y) = e^-(x, y)$. Now let $i \in I, a \in \psi_{u(i)}^{++}, b \in \varphi^-_i$ and suppose for all $c \in f^-(a, b)$ we have $\psi_{u(i)}(a, c)$. Then $\forall_u(\varphi)_{u(i)}(e^+(a), b)$ and $\varphi_i(e^+(a)[0], b)$, hence $\varphi_i(f^+(a), b)$, as desired.

(b) $\Rightarrow$ (a): Take $e^+(x) = \{ \lambda y. f^+(x) \}$ and $e^-(x, y) = f^-(x, y)$. Then let $j \in J, a \in \psi_j^{++}, b \in \forall_u(\varphi)_j^-$ and suppose for every $c \in e^-(a, b)$ we have $\psi_j(a, c)$. We want to show $\forall_u(\varphi)_j(e^+(a), b)$, i.e., $(\forall i \in u^{-1}(j)) \, \big( \, b \in \varphi_i^- \to \varphi_i(f^+(a), b) \, \big)$. But this is immediate from (b).

Validity of the Beck-Chevalley condition is immediate.
\end{proof}

\begin{lemm}{existquantification} For $u: I \to J$ and $\varphi: I \to \Sigma_{st}$ existential quantification is given by:
\begin{eqnarray*}
\exists_u(\varphi)_j^+ & = & \bigcup_{i \in u^{-1}(j)} \varphi_i^{++} \\
\exists_u(\varphi)_j^- & = & \bigcap_{i \in u^{-1}(j)} \varphi_i^{++} \to (\varphi^-_i)^* \\
\exists_u(\varphi)_j(a, b) & \Leftrightarrow & (\exists i \in u^{-1}(j)) \, ( \exists s \in a) \, \big( \, s \in \varphi_i^{++} \land (\forall c \in b(s)) \, \varphi_i(s, c) \, \big).
\end{eqnarray*}
Also the Beck-Chevalley condition holds.
\end{lemm}
\begin{proof}
Suppose $\varphi: I \to \Sigma_{st}$ and $\psi: J \to \Sigma_{st}$. We have to show the equivalence of the following two statements:
\begin{itemize}
\item[(a)] $\exists_u(\varphi) \vdash_J \psi$, i.e., there exist
\[ e^+ \in \bigcap_{j \in J} \exists_u(\varphi)_j^{++} \to \psi_j^{++} \qquad \mbox{ and } \qquad e^- \in \bigcap_{j \in J} \exists_u(\varphi)_j^{++} \times \psi_j^-  \to (\exists_u (\varphi)_j^-)^* \]
such that
\[ \forall j \in J, a \in \exists_u(\varphi)_j^{++}, b \in \psi_j^{-} \big( \, \forall c \in e^-(a, b) \exists_u(\varphi)_j(a, c) \big) \to \psi_j(e^+(a), b). \]
\item[(b)] $\varphi \vdash_I u^*\psi$, i.e., there exist 
\[ f^+ \in \bigcap_{i \in I} \varphi_i^{++} \to \psi_{u(i)}^{++} \qquad \mbox{ and } \qquad f^- \in \bigcap_{i \in I} \varphi_i^{++} \times \psi_{u(i)}^- \to (\varphi_i^-)^* \]
such that
\[ \forall i \in I, a \in \varphi_i^{++}, b \in \psi_{u(i)}^- \big( \, \forall c \in f^-(a, b) \, \varphi_i(a, c) \, \big)  \to \psi_{u(i)}(f^+(a), b). \]
\end{itemize}
(a) $\Rightarrow$ (b): Take $f^+(x) = e^+(\{x \})$ and $f^-(x, y) = e^-(\{x\}, y)[x] = \bigcup \{ z(x) \, : \, z \in e^-(\{x \}, y) \}$. Now let $i \in I, a \in \varphi_i^{++}, b \in \psi_{u(i)}^-$ and suppose for all $c \in f^-(a, b)$ we have $\varphi_i(a, c)$. Hence
\[ (\forall d \in e^-(\{a\}, b)) \, (\forall c \in d(a) ) \, \varphi_i(a, c). \]
Writing $j = u(i)$, we have $\{ a \} \in \exists_u(\varphi)_j^{++}$ and $b \in \psi_j^-$ and
\[ (\forall d \in e^-(\{a \}, b)) \, \exists_u(\varphi)_j (\{a\}, d). \]
Therefore $\psi_j(e^+(\{a \}), b)$, i.e., $\psi_{u(i)}(f^+(a), b)$.

(b) $\Rightarrow$ (a): Take $e^+(x) = \bigcup_{z \in x} f^+(z)$ and $e^-(x, y) = \{ \lambda z. f^-(z, y) \}$. Then let $j \in J, a \in \exists_u(\varphi)_j^{++}, b \in \psi_j^-$ and suppose for every $d \in e^-(a, b)$ we have $\exists_u(\varphi)_j(a, d)$. Concretely, this means that there is an $i \in u^{-1}(j)$ and an $s \in a$ such that $s \in \varphi_i^{++}$ and $\varphi_i  (s, c)$ for all $c \in f^-(s, b)$. This implies $\psi_{u(i)}(f^+(s), b)$, whence $\psi_j(e^+(a), b)$, because $\psi_j$ is upwards closed in the first component.

Validity of the Beck-Chevalley condition is immediate.
\end{proof}

\begin{lemm}{genpred}
The generic predicate is given by the identity on $\Sigma_{st}$.
\end{lemm}
\begin{proof} Clear.
\end{proof}

This completes the proof of \reftheo{istripos}.

\section{Open questions}

We have defined a new topos, but have not established any of its basic properties. Given the state of the art, we would conjecture the following:
\begin{enumerate}
\item Like the modified Diller-Nahm topos ${\bf DN}_m$ the topos we have defined is not 2-valued and its $\lnot\lnot$-sheaves do not coincide with the category of sets (see \cite{streicher06, biering08}).
\item First-order arithmetic in the topos we constructed is given by the $D_{st}$-interpretation of \cite{bergetal12} combined with using ${\rm HRO}$ as one's models of G\"odel's $T$.
\item As with the Herbrand topos, the functor $\nabla: \Sets \to {\bf Dst}$ preserves and refllects (at least) first-order logic, but not the natural numbers object. Hence $\nabla \NN$ is a model of nonstandard arithmetic in the $D_{st}$-topos (see \cite{berg12}).
\item As Jaap van Oosten has shown that the Herbrand topos ${\bf Her}$ is a subtopos of the modified realizability topos ${\bf Mod}$ and it is known that there is a connected geometric morphism from the modified Diller-Nahm topos ${\bf DN}_m$ to the modified realizability topos ${\bf Mod}$ (see \cite{biering08}), one would expect the $D_{st}$-topos to be a subtopos of ${\bf DN}_m$ and there to be a connected geometric morphism from it to the Herbrand topos. Indeed, one would expect there to be  a commuting square (pullback?) of toposes
\diag{ {\bf Dst} \ar[r] \ar[d] & {\bf DN}_m \ar[d] \\
{\bf Her} \ar[r] & {\bf Mod} }
in which the horizontal arrows are inclusions of toposes and the vertical ones are connected geometric morphisms.
\end{enumerate}

\bibliographystyle{plain} \bibliography{dsttopos}

\end{document}